\documentclass{article}
\usepackage{amsfonts}

\usepackage{amsmath,amssymb}

\newcommand{\R}{{\mathbb R}}

\newcommand{\bc}{\begin{center}}
\newcommand{\ec}{\end{center}}

\newcommand{\qed}{\enspace\vrule  height6pt  width4pt  depth2pt}
\newenvironment{proof}{\par\noindent{\bf Proof.}}{$\qed$\par\bigskip}
\newtheorem{theorem}{Theorem}[section]
\newtheorem{definition}[theorem]{Definition}
\newtheorem{lemma}[theorem]{Lemma}
\newtheorem{corollary}[theorem]{Corollary}
\newtheorem{proposition}[theorem]{Proposition}
\newtheorem{remark}[theorem]{Remark}
\newtheorem{example}[theorem]{Example}

\begin{document}

\date{}
\title{ Differential Calculus and Integration of Generalized Functions over
Membranes}
\author{Aragona, J., \and Fernandez, R., \and Juriaans, S.O., \and %
Oberguggenberger, M.}
\maketitle

\begin{abstract}
In this paper we continue the development of the differential calculus
started in \cite{afj}. Guided by the topology introduced in \cite{afj2} and 
\cite{afj1} we introduce the notion of membranes and extend the definition
of integrals, given in \cite{afj}, to integral defined on membranes. We use
this to prove a generalized version of the Cauchy formula and to obtain the
Goursat Theorem for generalized holomorphic functions. We also show that the
generalized transport equation can be solved giving an explicit solution.
\end{abstract}

\section{Introduction}

The theory of Colombeau generalized functions developed rapidly during the
last years. It has useful applications and gives new inside where the
classical theory does not (see \cite{mog}).

Having the algebraic theory (\cite{AJ}) as a starting point,
Aragona-Fernandez-Juriaans have developed a differential calculus which
allows to introduce most notions of differential calculus and geometry into
this context. Using the algebraic and differential theories,
Aragona-Fernandez-Juriaans (\cite{afj}) were able to generalize a result of 
\cite{AS} on the existence of solutions for linear PDE's. In \cite{asj}
these algebraic and differential theories were also used to study algebraic
properties of the algebra of Colombeau generalized functions.

In this paper we continue the development of the calculus started in \cite%
{afj}. After defining what we mean by an $n$-dimensional membrane we define
the integral of a generalized function over a membrane. We then proceed to
apply these notions and results. Among these applications are the Cauchy
Formula in the context of generalized holomorphic functions and examples
which show that for some linear operators the equation $L(u)=f$ can be
solved explicitly.

Basic references for the theory of Colombeau generalized numbers, functions
and their topologies are \cite{AB}, \cite{JC2}, \cite{GKOS}, \cite{gksv}, 
\cite{ku}, \cite{ko}, \cite{opd} and \cite{S2}.


\section{ Differential Calculus}

Let $\Omega\subset {\mathbb{R}}^n$ be an open subset, $\mathbf{I} =]0,1]$
and $\mathbf{I}_{\eta}:=]\,0,\eta[$ for each $\eta\in\mathbf{I}$. As usual, $%
{\mathbb{K}}$ denotes indistinctly ${\mathbb{R}}$ or ${\mathbb{C}}$. The
definitions of the algebra of the simplified generalized functions, $%
\mathcal{G}(\Omega)$, and the ring of the simplified generalized numbers, ${%
\overline{{\mathbb{K}}}}$, are the ones given in \cite{AJ}.

In this section we continue the theory developed in \cite{afj} and we shall
use results and notation from \cite{afj} and \cite{asj}. We remind that $%
\kappa : {\mathcal{G}}(\Omega ) \rightarrow \mathcal{C}^{\infty}({\tilde{%
\Omega}}_c , {\overline{{\mathbb{K}}}} )$ is the embedding introduced in 
\cite{afj} and that the function 
\begin{equation*}
({\mathcal{G}}(\Omega ))^p\ni (f_1,...,f_p)\longmapsto (\kappa
f_1,...,\kappa f_p)\in \mathcal{C}^{\infty}({\tilde{\Omega}}_c , {\overline{{%
\mathbb{K}}}} )^p
\end{equation*}
will also be denoted by $\kappa$.

The results here presented can be proved using the same arguments of their
classical analog observing that 
\begin{equation*}
\lim\limits_{x\to x_0}\displaystyle\,\frac{r(x)}{\alpha_{-\log ||x-x_0||}}%
=0\Longleftrightarrow\lim\limits_{x\to x_0}\displaystyle\,\frac{||r(x)||}{%
||x-x_0||}=0,
\end{equation*}
where $r:A\to{\overline{{\mathbb{R}}}}^{\,s}$ is a function, $A$ is an open
subset of ${\overline{{\mathbb{R}}}}^{\,m}$ and $x_0\in A$.

\begin{theorem}[Chain Rule]
Let $\,U$ be an open subset of $\,{\overline{{\mathbb{R}}}}^{\,m}\,$, $\,V$
an open subset of $\,{\overline{{\mathbb{R}}}}^{\,k}$, $\,f:U\to V$ a
function differentiable at $x_0\in U$ and $g:V\to{\overline{{\mathbb{R}}}}%
^{\,s}$ differentiable at $f(x_0)$. Then $g\circ f$ is differentiable at $x_0
$ and $D(g\circ f)(x_0)=Dg(f(x_0))\circ Df(x_0)$.
\end{theorem}

\begin{theorem}
Let $\,U\,$ and $\,V\,$ be open subsets of $\,{\overline{{\mathbb{R}}}}^{\,n}
$, $\,f:U\rightarrow V$ a function with inverse $g:V\rightarrow U$. If $\,f\,
$ is differentiable at $\,x_0\in U$, $\ det(Df(x_0))\in Inv({\overline{{%
\mathbb{R}}}} )$ and $\,g\,$ is continuous in $\,y_0:=f(x_0)\,$, then $\,g\,$
is differentiable at $\,y_0\,$.
\end{theorem}

We now announce the two most classical theorems of differential calculus.

\begin{theorem}[Inverse Function Theorem]
\label{teoinv} Let $\Omega$ be an open and convex subset of $\,{\mathbb{R}}^n
$, $f\in ({\mathcal{G}}(\Omega ))^n$ and $x_0\in{\tilde{\Omega}}_c$ such
that $\ det(D(\kappa(f))(x_0))\in Inv({\overline{{\mathbb{R}}}} )$. Then
there exist $\,U$ and $\,V$ open subsets of $\,{\overline{{\mathbb{R}}}}^n$
such that $x_0\in U$, $(\kappa(f))(x_0)\in V$ and $\kappa (f):U\to V$ is a $%
\mathcal{C}^{\infty}$-diffeomorphism.
\end{theorem}

\begin{theorem}[Implicit Function Theorem]
Let $\Omega$ be an open and convex subset of $\,{\mathbb{R}}^m\times{\mathbb{%
R}}^{k}$, $f\in ({\mathcal{G}}(\Omega ))^k$, $(x_0,y_0)\in{\tilde{\Omega}}_c$
such that $\kappa(f)(x_0,y_0)=0$ and $det(D_y(\kappa(f))(x_0,y_0))\in Inv({%
\overline{{\mathbb{R}}}} )$. Then there exist $\,U\subset{\overline{{\mathbb{%
R}}}}^{\,m}$ and $\,V\subset{\overline{{\mathbb{R}}}}^{\,k}$ with $%
(x_0,y_0)\in U\times V\subset {\tilde{\Omega}}_c$ such that for all $x\in U$
there is a unique $g(x)\in V$ with $(\kappa (f))(x,g(x))=0$. Moreover the
function $g:x\in U\mapsto g(x)\in V$ is $\mathcal{C}^{\infty}$ and $%
Dg(x)=[D_y(\kappa (f))(x,g(x))]^{-1}[-D_x(\kappa (f))(x,g(x)]$, for all $%
x\in U$.  \label{teoimpl}
\end{theorem}

Since the proofs do not differ much from the classical ones we omit them all.

\begin{remark}
Let $\Omega:=]-1,1[$ and $f:{\tilde{\Omega}}_c\to {\overline{{\mathbb{R}}}}$
the function defined by $f(x)=\alpha_1\,x$. Then $f=\kappa (F$, where $F=
[(\varepsilon,y)\in ]0,1]\times\Omega\mapsto\varepsilon y]$ has inverse $%
g=f^{-1}$, $g(x)=\alpha_{-1}x$, but there do not exist $V$, an open subset
of ${\mathbb{R}}$, and $h\in \mathcal{G}(V)$ such that $\kappa (h)=f^{-1}$.

The problem here is that there does not exists an open subset $V\subset {%
\mathbb{R}}$ such that $g(\widetilde{V}_c)\subset \widetilde{\Omega}%
_c\subset B^{\prime }_1(0)$. For if this were the case then, for $x_0\in
V\setminus \{0\}$ and $y_0=[(x_0)]$, we would have that $\| g(y_0)\| = e\|
y_0\|=e$ and hence $g(y_0)\notin B^{\prime }_1(0)$. Hence we may conclude
that there does not exists an open subset $V\subset {\mathbb{R}}$ and $h\in {%
\mathcal{G}}(V)$ such that $g=\kappa (h)$. Indeed, if they did exist, then
for the composition of $F$ and $h$ to make sense we must have that $h\in {%
\mathcal{G}}^*(V, ]-1, 1[)$ and hence, $g$ would be defined on $\widetilde{V}%
_c$ and so $Im(g)\subset \widetilde{\Omega}_c$. This proves that $F$ does
not have an inverse as a generalized function.

\typeout{The problem here is the image of $g$.  $Im(g)$ is not contained in
${\tilde{\Omega^{\prime}}}_c$, for any open subset $\Omega^{\prime}\subset
\R^n$, and so composition with another generalized function is not possible.
This is easily seen because for $x\in \Omega^{\prime}$ we have that $\|
g(x)\| = e\| x\|$ which must be less or equal then $1$ since we always have
that ${\tilde{\Omega^{\prime}}}_c \subset B_1(0)$.  But
${\tilde{\Omega^{\prime}}}_c\bigcap S_1(0)\neq \emptyset$ and hence the
composition $f\circ g$ is not possible.}

In the next section we introduce a way in which an inverse of the function
defined above exists without appealing to the inverse defined in ${\overline{%
{\mathbb{K}}}}$.
\end{remark}


\section{Integration on Membranes}

Given $\,x_0\in{\overline{{\mathbb{K}}}}$ and $0<r\in{\mathbb{R}}\,$ we
define $V_r[x_0]:=\{x\in{\overline{{\mathbb{K}}}}\,|\,|x-x_0|\leq\alpha_r\}$%
, where the definition of absolute value was introduced in \cite{asj} and $%
\alpha_r$ was defined in \cite{AJ}. It is proved in \cite{afj2} and \cite%
{afj1} that $\{V_r[x_0]\,|\,0<r\in{\mathbb{R}}\,,\,x_0\in{\overline{{\mathbb{%
K}}}}\}$ is a basis of a topology in ${\overline{{\mathbb{K}}}}$ which
coincides with Scarpalezos' sharp topology. It is easy to verify that $x\in
V_r[x_0]$ if and only if there exist representatives $(x_{\varepsilon})\,,%
\,(x_{0\varepsilon})$ of $x$ and $x_0$ respectively such that $%
x_{\varepsilon}\in B_{\varepsilon^r}(x_{0\varepsilon})$, $\forall
\varepsilon\in \mathbf{I}$. So if $x_0\in {\tilde{\Omega}}_c$, $%
(x_{0\varepsilon})$ is a representative of $x_0$ and $(M_{\varepsilon}):=
(B_{\varepsilon^r}({x_{0\varepsilon}})\cap\Omega)$, then $%
V_r[x_0]=\{[(x_{\varepsilon})]|\ x_{\varepsilon}\in M_{\varepsilon}\}$.
Based on this fact we introduce the notion of membranes which will allow us
to effectively integrate generalized functions.

We first start by defining what the subsets are over which we will be
integrating generalized functions and then define how to integrate over
these sets.

\begin{definition}
\label{membrana} We denote by ${\mathcal{P}}(\Omega )_M$ the family of
subsets $(M_{\varepsilon})$ such that

\begin{enumerate}
\item $\exists K\subset \Omega$ a compact subset and $\eta\in\mathbf{I}$
such that $M_{\varepsilon} \subset K$ for $\varepsilon \in \mathbf{I}_\eta $;

\item the characteristic function of $\,M_{\varepsilon}\,$ is Riemann
integrable for all $\,\varepsilon \in\mathbf{I}_\eta$.
\end{enumerate}
\end{definition}

\noindent \textit{Any element of ${\mathcal{P}}(\Omega )_M$ is called a $n-$%
dimensional pre-membrane in $\Omega$.}

\vspace{0.3cm} \noindent Note that $\ 1.\ $ implies that $%
\{[(x_{\varepsilon})]|\ x_{\varepsilon}\in
M_{\varepsilon}\,,\forall\,\varepsilon\in\mathbf{I}\}\subset {\tilde{\Omega}}%
_c$. \vspace{0.3cm}

\begin{definition}
\label{history} Let $\gamma=(\gamma_{\varepsilon})$ be a family of elements
of $\mathcal{C}^1([0,1],{\mathbb{R}}^n)$. It is called a $n$-dimensional
history or just history if $(\gamma_{\varepsilon}([0,1]))$ is a pre-membrane
and there are $\,N\in{\mathbb{N}}\,,\,c>0$ and $\,\eta\in \mathbf{I}$ such
that $|\gamma_{\varepsilon}^{\,\prime}(t)|\leq
c\varepsilon^{-N}\,,\forall\,\varepsilon\in\mathbf{I}_{\eta}$ and $\,t\in
[0,1]$. The pre-membrane $(\gamma_{\varepsilon}([0,1]))$ is denoted by $%
\gamma^*$.
\end{definition}

\vspace{0.3cm} Two elements $(M_{\varepsilon})$ and $(M^{\prime}_{%
\varepsilon})$ of ${\mathcal{P}}(\Omega )_M$ are said to be equivalent if
there exists a null-function $(\Psi_{\varepsilon})\in {\mathcal{N}}(\Omega;{%
\mathbb{R}}^n)$ such that the function $\phi$ defined on $\mathbf{I}%
\times\Omega$ by ${\phi}(\varepsilon,x):=x+\Psi_{\varepsilon}(x)$ satisfies $%
{\phi}(\varepsilon
,M_{\varepsilon})=M^{\prime}_{\varepsilon}\,,\,\forall\,\varepsilon\in\mathbf%
{I}$. This clearly defines an equivalence relation on ${\mathcal{P}}(\Omega
)_M$. We denote the quotient space by ${\mathcal{P}}(\Omega )_{M}/_{\sim}\,$
and call its elements \textit{$n$-dimensional membranes} in $\Omega$ or just
membranes. It is easy to verify that if $(M_{\varepsilon})$ and $%
(M^{\prime}_{\varepsilon})$ are equivalent and $\gamma=(\gamma_{\varepsilon})
$ is a history such that $\gamma^*=(M_{\varepsilon})$, then $%
(\beta_{\varepsilon})$ is a history and $\beta^*=(M^{\prime}_{\varepsilon})$%
, where $\beta_{\varepsilon}(t):=\phi(\varepsilon,\gamma_{\varepsilon}(t))\,$%
.

\vspace{0.3cm} Note that if $\,(M_{\varepsilon})$ and $\,(M^{\prime}_{%
\varepsilon})$ are equivalent pre-membranes then

\vskip-0.3cm  
\begin{equation*}
\{[(x_{\varepsilon})]\,|\, \ x_{\varepsilon}\in
M_{\varepsilon}\,,\,\forall\varepsilon\in\mathbf{I}\}=\{[(y_{\varepsilon})]%
\,|\, \ y_{\varepsilon}\in M^{\prime}_{\varepsilon}\,,\,\forall\varepsilon\in%
\mathbf{I}\}\,.
\end{equation*}

\noindent  So we can define the function 
\begin{equation*}
j:\mathcal{P}(\Omega)_{M}/_{\sim}\ \ni
[(M_{\varepsilon})]\longmapsto\{[(x_{\varepsilon})]\,|\, \
x_{\varepsilon}\in M_{\varepsilon}\,,\,\forall\varepsilon\in\mathbf{I}%
\}\subset{\tilde{\Omega}}_c\,.
\end{equation*}

\vspace{0.3cm} Since $vol(M_{\varepsilon})$ is uniformly bounded for small $%
\,\varepsilon$, we can define $vol(X)$ by $vol(X):=[\varepsilon \rightarrow
vol(M_{\varepsilon})].$

From here on we shall write $\,X=[(M_{\varepsilon})]$ instead of $%
\,X=j(\,[(M_{\varepsilon})]\,)\,.$ If we drop condition 2. of definition~\ref%
{membrana} we will speak of a pseudo-membrane, i.e., a family of subsets
satisfying only the first condition of definition~\ref{membrana} will be
called a pseudo-membrane. The image by $j$ of a pseudo-membrane shall still
be called a pseudo-membrane.

\begin{lemma}
\label{closed} Let $(M_{\varepsilon} )$ be a pseudo-membrane. Then $%
j([(M_{\varepsilon} )]) = j([({\overline{M_{\varepsilon}}} ))]$, where ${%
\overline{M_{\varepsilon}}}$ is the topological closure of $M_{\varepsilon}$.
\end{lemma}

\begin{proof}
It is enough to prove that $j([({\overline{M_{\varepsilon}}})]) \subset j([({%
M_{\varepsilon}})])$. Choose $x\in j([({\overline{M_{\varepsilon}}})])$;
then $x=[(x_{\varepsilon})]$, with $x_{\varepsilon}\in {\overline{%
M_{\varepsilon}}}$. For each $x_{\varepsilon}$ we may choose $%
y_{\varepsilon}\in M_{\varepsilon}$ such that $|x_{\varepsilon}-y_{%
\varepsilon}|<exp(-\frac{1}{\varepsilon})$. Since $(exp(-\frac{1}{\varepsilon%
}))$ is a null-element we are done.
\end{proof}

If $x,y$ are points in ${\mathbb{K}}^n$ we define the \textit{generalized
distance} between them as $d(x,y)=dist(x,y):=[\varepsilon \rightarrow
dist(x_{\varepsilon},y_{\varepsilon})]$. This is a well defined element of ${%
\overline{{\mathbb{R}}}}$. We recall also that in \cite{AJ} it is proved
that if $w\in {\overline{{\mathbb{K}}}}$ is a non-zero element then there
exists an idempotent $e\in {\overline{{\mathbb{K}}}}$ and $r\in {\mathbb{R}}$
such that $e\cdot |w|>e\cdot \alpha_r$. If $x_0\in {\overline{{\mathbb{K}}}}%
^n$ then $V_r[x_0]=\{x\in {\overline{{\mathbb{K}}}}^n \ |\
d(x,x_0)<\alpha_r\}$. In \cite{afj1} it is proved that these sets are a
basis of neighborhood of the sharp topology of ${\overline{{\mathbb{K}}}}^n$%
. We will use these observations in what follows.

\begin{proposition}
Pseudo-membranes are closed in the sharp topology. Moreover, if $%
M=j([(M_{\varepsilon})])$ and $M_{\varepsilon}$ is convex, for small $%
\epsilon$, then $M$ is not open.
\end{proposition}

\begin{proof}
Let $M$ be a pseudo-membrane and choose $[(x_{\varepsilon})] = x\not\in M$.
We may suppose, by Lemma~\ref{closed}, that $M=j([(M_{\varepsilon})])$, with
all $M_{\varepsilon}$ closed. It follows that $dist(x,M) = d:=[\varepsilon
\rightarrow dist(x_{\varepsilon},M_{\varepsilon})]$ is a non-zero element of 
${\overline{{\mathbb{K}}}}$ and hence there exists an idempotent $e\in {%
\overline{{\mathbb{K}}}}$ and $r\in {\mathbb{R}}$ such that $e\cdot
\alpha_r<e\cdot d$. Now let $r<s$ and $y\in V_s[x]\bigcap M$. Then, $d\leq
dist(x,y)<\alpha_s<\alpha_r$ and thus $e\cdot \alpha_r<e\cdot d\leq e\cdot
dist(x,y)<e\cdot \alpha_s<e\cdot \alpha_r$, a contradiction.

Let $x\in M$ be an interior point whose representative $(x_{\varepsilon})$
satisfies $x_{\varepsilon}\in \partial M_{\varepsilon}$. Since $x$ is an
interior point, there exists $r\in {\mathbb{R}}$ such that $V_r[x]\subset M$%
. Since the $M_{\varepsilon}$'s are convex, we may choose points $%
y_{\varepsilon}$ of norm $1$ such that $z_{\varepsilon}:=x_{\varepsilon}+{%
\varepsilon}^r y_{\varepsilon}$ satisfies $d(z_{\varepsilon}, {\overline{%
M_{\varepsilon}}})\geq \varepsilon^r$. Set $z=[(z_{\varepsilon})]$; then $%
dist(z,M)\geq\alpha_r$ and hence $z\not \in M$. On the other hand $z\in
V_r[x]\subset M$, a contradiction.
\end{proof}

As a corollary we have the following result.

\begin{corollary}
Let $(M_n)$ be a decreasing sequence of pseudo-membranes with diameters
tending to zero. Then $\displaystyle\bigcap_{n\in\mathbf{N}} M_n$ consists
of a single point.
\end{corollary}

\begin{proof}
This is clear since ${\overline{{\mathbb{K}}}}^n$ is a complete metric space
and pseudo-membranes are closed.
\end{proof}

Observe that the element $x=[(\frac{1}{-ln \varepsilon})]\approx 0$ and has
norm $1$. It follows that $B_1(0)$ is a proper subset of ${\overline{{%
\mathbb{K}}}}_0=\{x\in {\overline{{\mathbb{K}}}} \ | \ x\approx 0\}$.
Observe also that $M=B_1(0)$ is not a membrane: in fact suppose that $%
M=[(M_{\varepsilon})]$ and let $N_{\varepsilon}$ be the convex hull of $%
M_{\varepsilon}$. As proved above, we may suppose that $M_{\varepsilon}$ is
closed for all $\varepsilon$. As is easily seen, $N=[(N_{\varepsilon})]$ is
contained in the convex hull of $M$. Since $M$ is a sub ring of ${\overline{{%
\mathbb{K}}}}$, it follows that it equals its convex hull. It follows that $%
M=N$ and hence is not open because all the $N_{\varepsilon}$'s are convex, a
contradiction.

\begin{example}
\label{esferas}

Let $x\in {\widetilde{{\mathbb{R}}^n}}_c$, $r\in Inv({\overline{{\mathbb{R}}}%
})$ and let $(x_{\varepsilon})\,,\,\ (x^{\prime}_{\varepsilon})$ be
representatives of $\,x$ and $\ (r_{\varepsilon})\,,\,
(r^{\prime}_{\varepsilon} )$ be representatives of $\,r$. Consider the
pre-membranes $(M_{\varepsilon}:=B_{r_{\varepsilon}}(x_{\varepsilon}))$ and $%
(M^{\prime}_{\varepsilon}:=B_{r^{\prime}_{\varepsilon}}(x^{\prime}_{%
\varepsilon}))$. Define $(\Psi_{\varepsilon})$ by

\vskip-0.3cm 
\begin{equation*}
\Psi_{\varepsilon}(y):={\frac{r^{\prime}_{\varepsilon}-r_{\varepsilon}}{{%
r_{\varepsilon}}}}(y-x_{\varepsilon})+{\frac{r_{\varepsilon}-r^{\prime}_{%
\varepsilon}}{{r_{\varepsilon}}}}(x^{\prime}_{\varepsilon}-x_{\varepsilon})%
\,.
\end{equation*}

\vskip-0.1cm \noindent Then $(\Psi_{\varepsilon})\in {\mathcal{N}}({\mathbb{R%
}}^n )$ and $\phi(\varepsilon,w):= w+\Psi_{\varepsilon}(w)$ , $%
\forall\,(\varepsilon,w)\in\mathbf{I}\times{\mathbb{R}}^n$, satisfies ${\phi}%
(\varepsilon
,M_{\varepsilon})=M^{\prime}_{\varepsilon}\,,\,\forall\,\varepsilon\in\mathbf%
{I}$. Hence $(M_{\varepsilon})$ and $(M^{\prime}_{\varepsilon})$ are
equivalent. When $r=\alpha_s$ then $j([(M_{\varepsilon})])$ is just $V_s[x]$%
. Its volume is $\ vol(V_s[x])=\pi\alpha_s^2=\pi\alpha_{2s}$.
\end{example}

\vspace{0.3cm} For this reason we call $V_s[x]$ a \textit{generalized ball}
whose center is $\ x\ $ and whose radius is $\ \alpha_s$. By a \textit{%
generalized sphere} we shall mean a set of the form $\{x\in {\tilde{\Omega}}%
_c| \ \|x-x_0\|=\alpha_s\}$, for some $s\in {\mathbb{R}}$ and $x_0\in {%
\tilde{\Omega}}_c$.

\vspace{0.3cm}  We are now in position to define a notion of integration of
generalized functions that is consistent with the differential calculus we
have developed so far. For this we need the following result.

\begin{proposition}
\label{welldef} Let $f\in {\mathcal{G}}(\Omega )$, $(f_{\varepsilon})$ a
representative of $\ f$ and $\lambda$ the Lebesgue measure on ${\mathbb{R}}^n
$.

\begin{enumerate}
\item If $\,(M_{\varepsilon})$ is a pre-membrane, then the function $%
\varepsilon\longmapsto\int_{M_{\varepsilon}}f_{\varepsilon}\,d\lambda$ is
moderate (is null if $(f_{\varepsilon})$ is null).

\item If $\,(M_{\varepsilon})\,$ and $\,(M^{\prime}_{\varepsilon})\,$ are
equivalent pre-membranes, then 
\begin{equation*}
[\,\varepsilon\longmapsto
\int_{M_{\varepsilon}}f_{\varepsilon}\,d\lambda]=[\varepsilon\longmapsto
\int_{M^{\prime}_{\varepsilon}}f_{\varepsilon}\,d\lambda]\,.
\end{equation*}

\item If $\,(M_{\varepsilon})\,$ and $\,(M^{\prime}_{\varepsilon})\,$ are
equivalent pre-membranes and $(g_{\varepsilon})$ is a representative of $f$,
then $\,[\varepsilon\longmapsto
\int_{M_{\varepsilon}}f_{\varepsilon}\,d\lambda]=[\varepsilon\longmapsto
\int_{M^{\prime}_{\varepsilon}}g_{\varepsilon}\,d\lambda]\,.$
\end{enumerate}
\end{proposition}

\begin{proof}
The assertion 1. is obvious. For 2. and 3., let $\Psi=[(\Psi_{\varepsilon})]
\in {\mathcal{N}}(\Omega;{\mathbb{R}}^n)$ and ${\phi}(\varepsilon,x):=
x+\Psi_{\varepsilon}(x)$, $\forall\,(\varepsilon,x)\in\mathbf{I}\times\Omega$%
, such that ${\phi}(\varepsilon
,M_{\varepsilon})=M^{\prime}_{\varepsilon}\,,\,\forall\,\varepsilon\in\mathbf%
{I}$. Denote by $D{\phi}$ the Jacobian matrix of ${\phi}$ and let $J$ be its
determinant. Then we have that $J=1+\tau$ with $\tau \in {\mathcal{N}}%
(\Omega )$ and {\small 
\begin{equation*}
|\int_{M_{\varepsilon}}{f}_{\varepsilon}d\lambda-\int_{M^{\prime}_{%
\varepsilon}}{f}_{\varepsilon}d\lambda|\leq \int_{M_{\varepsilon}}|\,{f}%
_{\varepsilon}(x)-{f}_{\varepsilon}({\phi}(\varepsilon ,x))|\,dx+
\int_{M_{\varepsilon}}|{f}_{\varepsilon}({\phi}(\varepsilon ,x))\tau
(\varepsilon , x)|\,dx\,.
\end{equation*}%
}  Choosing a compact subset $K$ containing $\cup_{\varepsilon\in\,\mathbf{I}%
}(M_{\varepsilon}\cup M^{\prime}_{\varepsilon})$,  we can find $%
x_1,...,x_s\in K$ and $r_1,...,r_s>0\,$ such that  
\begin{equation*}
\,K\subset\displaystyle L_1:=\cup_{1\leq j\leq s}B_{r_j}^{\prime
}(x_j)\subset L:=\cup_{1\leq j\leq s}B_{2r_j}^{\prime }(x_j)\subset\Omega\,.
\end{equation*}
As $\Psi\in \mathcal{N}(\Omega;{\mathbb{R}}^n)$ there is $\eta_1\in\mathbf{I}
$ such that $\phi_{\varepsilon}(B_{r_j}^{\prime }(x_j))\subset
B_{2r_j}^{\prime }(x_j)\,,\,\forall\varepsilon\in\mathbf{I}_{\eta_1}$. Since 
$\,f\,$ is moderate there are $\,N\in{\mathbb{N}}\,$, $c>0$ and $\,\eta\in%
\mathbf{I}_{\eta_1}\,$ such that $\,\max\{\,||\nabla {f}_{\varepsilon}(x)||%
\,,\,|f_{\varepsilon}(x)|\,\}\leq c\varepsilon^{-N}\,\,,\,\forall\,x\in
L\,,\,\forall\,\varepsilon\in\mathbf{I}_{\eta}\,$. Hence, by the Mean Value
Theorem, noting that $B_{r_j}^{\prime }(x_j)\subset B_{2r_j}^{\prime }(x_j)$
and $B_{2r_j}^{\prime }(x_j)$ are convex for all $\,1\leq j\leq s$, we
conclude that 
\begin{eqnarray*}
\int_{M_{\varepsilon}}|\,{f}_{\varepsilon}(x)-{f}_{\varepsilon}(\phi(%
\varepsilon ,x))|\,dx&\leq&\sum_{1\leq j\leq s}\int_{B_{r_j}^{\prime
}(x_j)}|\,{f}_{\varepsilon}(x)-{f}_{\varepsilon}({\phi}(\varepsilon ,x))|\,dx%
\cr &\leq&\sum_{1\leq j\leq s}\int_{B_{r_j}^{\prime
}(x_j)}c\varepsilon^{-N}|\Psi(\varepsilon,x)|\,dx\,,\cr
\end{eqnarray*}
$\,\forall\varepsilon\in\mathbf{I}_{\eta}$. Using that $\Psi\in \mathcal{N}%
(\Omega;{\mathbb{R}}^n)\,$, $\,\tau \in {\mathcal{N}}(\Omega )$ and $\,|{f}%
_{\varepsilon}({\phi}(\varepsilon ,x))\tau (\varepsilon , x)|\leq
\varepsilon^{-N}|\tau(\varepsilon,x)|\,,\,\forall\,\varepsilon\in\mathbf{I}%
_{\eta}\,,\,\forall\,x\in L_1$, we conclude that 2. holds. For 3. note that 
\begin{equation*}
|\int_{M_{\varepsilon}}{f_{\varepsilon}}\,d\lambda-\int_{M^{\prime}_{%
\varepsilon}}{g_{\varepsilon}}\,d\lambda|\leq|\int_{M_{\varepsilon}}{%
f_{\varepsilon}}\,d\lambda-\int_{M^{\prime}_{\varepsilon}}{f_{\varepsilon}}%
\,d\lambda|+\int_{M^{\prime}_{\varepsilon}}|f_{\varepsilon}-g_{\varepsilon}|%
\,d\lambda\,.
\end{equation*}
The result now readily follows from the others assertions.
\end{proof}

The proposition above guarantees that the following definition makes sense.

\begin{definition}
Let $f=[(f_{\varepsilon})]\in {\mathcal{G}}(\Omega )$, $M=[(M_{\varepsilon})]
$ a membrane of ${\overline{{\mathbb{R}}}}^n$ and $\lambda$ the Lebesgue
measure on ${\mathbb{R}}^n$. The generalized number

\vskip-0.3cm  
\begin{equation*}
\int_Mf:=[\varepsilon\longmapsto\int_{M_{\varepsilon}}f_{\varepsilon}\,d%
\lambda\,]
\end{equation*}
is called the integral of $\,f$ over the membrane $\,M$.
\end{definition}

\begin{definition}
Let $\Omega\subset{\mathbb{R}}^n$, $\,f=[(f_{\varepsilon})]\in ({\mathcal{G}}%
(\Omega ))^n$ and $\,\gamma=(\gamma_{\varepsilon})$ a history.

\begin{enumerate}
\item If ${\mathbb{K}}={\mathbb{R}}$, then the generalized number

\vskip-0.3cm 
\begin{equation*}
\int_{\gamma}\,f\,d\gamma:=[\varepsilon\longmapsto\int_0^1\,<
f_{\varepsilon}(\gamma_{\varepsilon}(t))\,|\,\gamma_{\varepsilon}^{%
\prime}(t)>\,dt]\,
\end{equation*}
is called the (line) integral of $\,f\,$ along $\gamma$, where $%
\,<\,\cdot\,|\,\cdot\,>\,$ denotes the standard inner product of $\,{\mathbb{%
R}}^n$.

\item If $n=2$, $\Omega\subset {\mathbb{C}}$, $\,{\mathbb{K}}={\mathbb{C}}$
and $f\in{\mathcal{G}}(\Omega )$, then the generalized number

\vskip-0.3cm 
\begin{equation*}
\int_{\gamma}\,f\,dz:=[\varepsilon\longmapsto\int_0^1\,
f_{\varepsilon}(\gamma_{\varepsilon}(t))\gamma_{\varepsilon}^{\prime}(t)\,dt%
\,]
\end{equation*}
is called the integral of $\,f\,$ along $\gamma$.
\end{enumerate}
\end{definition}

It is easy to verify that these definitions make sense.

Note that if $n=1$ and $([a_{\varepsilon},b_{\varepsilon}])$ is a membrane,
then the history $\gamma=(\gamma_{\varepsilon})$ where $\gamma_{%
\varepsilon}(t):=a_{\varepsilon}+t(b_{\varepsilon}-a_{\varepsilon})\,,\,%
\forall\,t\in\,[0,1]$, can be identified with the element $%
(a,b):=([(a_{\varepsilon})],[(b_{\varepsilon})])\in\widetilde{{\mathbb{R}}^2}%
_c$ and $\int_{([a_{\varepsilon},b_{\varepsilon}])}\,f=\int_{\gamma}\,f\,d%
\gamma=\int_a^b\,f$, where the last integral is the one given in \cite[%
section 4]{afj}. Conversely an element $(c,d):=([c_{\varepsilon},d_{%
\varepsilon}])\in \widetilde{{\mathbb{R}}^2}_c$ defines a history $%
\beta=(\beta_{\varepsilon})$ such that $\int_{\beta}\,f\,d\beta=\int_c^d\,f$
(it is enough to define $\beta_{\varepsilon}=c_{\varepsilon}+t(d_{%
\varepsilon}-c_{\varepsilon})\,,\,$ if $c_{\varepsilon}\leq d_{\varepsilon}$
and $\beta_{\varepsilon}=d_{\varepsilon}+t(d_{\varepsilon}-c_{\varepsilon})%
\,,\,$ if $c_{\varepsilon}> d_{\varepsilon}$, $\forall\,t\in\,[0,1]$. In
this case, the definition given here agrees with the one given in \cite[%
section 4]{afj}.


\section{\textbf{Calculus on Membranes}}

In this section we give some applications of the theory developed in the
previous section.

\begin{proposition}
\label{desig} Let $\,f\in {\mathcal{G}}(\Omega )$, $M$ a membrane of $\,{%
\overline{{\mathbb{R}}}}^n$ and $\lambda$ the Lebesgue measure on ${\mathbb{R%
}}^n$. Then we have:

\begin{enumerate}
\item There exists $x_0\in {\tilde{\Omega}}_c$ such that $\int_M
f=vol(M)f(x_0)$.

\item There exists $r\in {\mathbb{R}}$ such that $|\int_Mf\,d\lambda|\leq
vol(M)\alpha_r$.
\end{enumerate}
\end{proposition}

\begin{proof}
The first item follows readily from its classical analog and the second one
follows from the first one and the definition of $\ {\overline{{\mathbb{K}}}}
$.
\end{proof}

In what follows $e_i$ will stand for $(0,\cdots ,1,\cdots )\in {\overline{{%
\mathbb{K}}}}^n$ and $\langle \cdot|\cdot \rangle$ will denote the standard
bilinear form induced by the standard inner product of $\ {\mathbb{K}}^n$.

It is easily seen that if $f:{\tilde{\Omega}}_c \rightarrow {\overline{{%
\mathbb{K}}}}$ is differentiable at $x_0\in {\tilde{\Omega}}_c$, then there
exist a continuous function $\phi$ with $\phi (x_0)=0$ and such that $%
f(x)-f(x_0)= \langle\nabla f(x_0)|(x-x_0)\rangle +\phi
(x)\alpha_{-\log\,||x-x_0||}.$ From this and the fact that if $\gamma :%
\tilde{{\mathbb{K}}}_c\rightarrow {\tilde{\Omega}}_c$ is differentiable at $%
t_0$ then {\small 
\begin{eqnarray*}
{\frac{||\gamma(t)-\gamma(t_0)||}{{||t-t_0||}}}&=&{\frac{||\gamma(t)-%
\gamma(t_0)-\gamma^{\prime }(t_0)(t-t_0) +\gamma^{\prime }(t_0)(t-t_0)||}{{%
||t-t_0||}}}\cr\cr\cr &=&\left |\left|{\frac{\gamma(t)-\gamma(t_0)-\gamma^{%
\prime }(t_0)(t-t_0)+\gamma^{\prime }(t_0)(t-t_0)}{{\alpha_{-\log\,||t-t_0||}%
}}}\right |\right |\cr\cr\cr &\leq&\max\left\{\left |\left |{\frac{%
\gamma(t)-\gamma(t_0)-\gamma^{\prime }(t_0)(t-t_0)}{{\alpha_{-\log%
\,||t-t_0||}}}}\right |\right |\,,\,\left |\left |{\frac{\gamma^{\prime
}(t_0)(t-t_0)}{{\alpha_{-\log\,||t-t_0||}}}}\right |\right |\right\}\cr\cr%
\cr &=&\max\left\{ \left |\left |{\frac{\gamma(t)-\gamma(t_0)-\gamma^{\prime
}(t_0)(t-t_0)}{{\alpha_{-\log\,||t-t_0||}}}}\right |\right
|\,\,,\,\,||\gamma^{\prime }(t_0)||\right\}\cr
\end{eqnarray*}%
}

We deduce, using standard techniques of differential calculus, the chain
rule for curves:

\begin{theorem}
\label{curves}

Let $f:{\tilde{\Omega}}_c \rightarrow {\overline{{\mathbb{K}}}}$ and $\gamma
:\tilde{{\mathbb{K}}}_c\rightarrow {\tilde{\Omega}}_c$. If $\gamma$ is
differentiable at $t_0$ and $f$ is differentiable at $x_0=\gamma (t_0)$ then 
$F:=f\circ \gamma$ is differentiable at $t_0$ and $F^{\prime}(t_0)=\langle
\nabla f(\gamma(t_0))|\gamma^{\prime}(t_0)\rangle$.
\end{theorem}

A history $\gamma=(\gamma_{\varepsilon})$ is closed if $\gamma_{\varepsilon}$
is a closed curve for small $\varepsilon$ and $\gamma$ is simple if for
small $\varepsilon$ we have that $\gamma_{\varepsilon}$ is a simple curve.
In an obvious way we define positively and negatively oriented histories. We
say that $\gamma$ is contractible if $\gamma_{\varepsilon}$ is homotopic to $%
\,0\,$ for small $\varepsilon$.

If $\Omega \subset {\mathbb{R}}^2$ and $f=[(f_{\varepsilon})]\in ({\mathcal{G%
}}(\Omega ))^2$ then define the generalized function $rot(f):=[\,(\,x\in%
\Omega\mapsto rot\,f_{\varepsilon}(x)\,)\,].$ This is obviously a well
defined element of $\,({\mathcal{G}}(\Omega ))^3$. We can now state the 
\textit{Generalized Green Theorem}.

\vspace{0.3cm}

\begin{theorem}[Green's theorem]
\label{green}

Let $\Omega \subset {\mathbb{R}}^2$, $\lambda$ the Lebesgue measure on ${%
\mathbb{R}}^2$, $\,f\in ({\mathcal{G}}(\Omega ))^2$ and $\gamma=(\gamma_{%
\varepsilon})$ a closed, simple, contractible and positively oriented
history. If $\,M=[(\gamma_{\varepsilon}([0,1]))]\,$, then 
\begin{equation*}
\int_{\gamma}fd\gamma =\int_{M}\langle rot(f)\,|\,e_3\rangle\,d\lambda\,.
\end{equation*}
\end{theorem}

Most other theorems of classical differential calculus can now, in a very
natural way, be translated to this context. Since we have also the notion of
a generalized manifold these results should also be extended to generalized
manifolds.


\section{\textbf{The Generalized Cauchy Formula}}

In this section $\Omega\subset{\mathbb{C}}$ and ${\mathbb{K}} ={\mathbb{C}}$%
. Let $\gamma =(\gamma_{\varepsilon})$ be a closed, simple, contractible
history and $z_0=[(z_{0\varepsilon})]\in {\overline{{\mathbb{C}}}}$ such
that $\,z_{0\varepsilon}$ belong to the bounded connected component of $%
\,\Omega\setminus\gamma_{\varepsilon}([0,1])\,,\,\forall\,\varepsilon\in%
\mathbf{I}$. As $\gamma^*$ is a pre-membrane we can define the generalized
number 
\begin{equation*}
d(z_0,\gamma^*):=[\varepsilon\mapsto
d(z_{0\varepsilon},\gamma_{\varepsilon}([0,1]))]
\end{equation*}
where $d(z_{0\varepsilon},\gamma_{\varepsilon}([0,1]))$ is the distance of $%
z_{0\varepsilon}$ to the set $\gamma_{\varepsilon}([0,1])\,.$ Note that, if $%
d(z_0,\gamma)\in Inv({\overline{{\mathbb{K}}}})$, then $z-z_0\in Inv({%
\overline{{\mathbb{C}}}})$ for all $z\in [\gamma^*]\,$.

Let $\Omega$ be a simply connected domain and $\,f\in\mathcal{HG}(\Omega)$.
In \cite{afj} we proved that any such $\ f$ has a convergent Taylor series.
In \cite{op} is proved that $\ f$ has a representative $({f}_{\varepsilon})$
such that ${f}_{\varepsilon}$ is holomorphic for all $\varepsilon\in\mathbf{I%
}$. Using this and the classical Cauchy Theorem we get the \textit{G-Cauchy
Formula}.

\begin{theorem}[G-Cauchy Formula]
\label{cauchy}  Let $\Omega$ be a simply connected set, $\gamma$ be a
closed, simple, contractible, positively oriented history and $%
z_0=[(z_{0\varepsilon})]\in {\overline{{\mathbb{C}}}}$ such that $%
\,z_{0\varepsilon}$ belong to the bounded connected component of $%
\,\Omega\setminus\gamma_{\varepsilon}([0,1])\,,\,\forall\,\varepsilon\in 
\mathbf{I}\,$. If $d(z_0,\gamma^*)\in Inv({\overline{{\mathbb{C}}}})$ and $%
\,f=[(f_{\varepsilon})]\in\mathcal{HG}(\Omega)$ , then 
\begin{equation*}
\,\,(\kappa (f))(z_0)=\frac{1}{2\pi i}\int_{\gamma}\frac{f(z)}{z-z_0}%
dz:=[\varepsilon\mapsto\frac{1}{2\pi i}\int_{\gamma_{\varepsilon}}{\frac{%
f_{\varepsilon}(z)}{{z-z_{0\varepsilon}}}}\,dz\,]\,.
\end{equation*}
\end{theorem}

Now we shall prove the Goursat Theorem in this context. In \cite{afj} this
theorem was proved with the condition that $f$ is sub-linear.

\begin{theorem}[Goursat Theorem]
\label{goursat} Let $\,f\in\mathcal{HG}(\Omega)$ . Then $\,\kappa (f)$ is
analytic in ${\tilde{\Omega}}_c$.
\end{theorem}

\begin{proof}
Let $z_0=[(z_{0\varepsilon})]\in{\tilde{\Omega}}_c$ and $f=[(f_{%
\varepsilon})]$. Then there are $\eta_1\in\mathbf{I}$ and $\,K\subset\Omega$
compact sets such that $z_{0\varepsilon}\in K\,,\,\forall\varepsilon\in%
\mathbf{I}_{\eta_1}\,$. From this we can choose $R>0$ such that $%
B_R(z_{0\varepsilon})\subset\Omega\,,\,\forall\,\varepsilon\in\mathbf{I}%
_{\eta_1}$. Let $\,\rho<r/4<R/2$ and $\,\eta_2\in\mathbf{I}_{\eta_1}$ such
that $\varepsilon^{\rho}<\rho\,,\,\forall\,\varepsilon\in\mathbf{I}%
_{\eta_2}\,$. Define $\,\gamma_{\varepsilon}(t):=z_{0\varepsilon}+re^{2\pi
it}\,,\,\forall\,t\in\,[0,1]\,,\,\forall\,\varepsilon\in\mathbf{I}_{\eta_2}$
and $\,\gamma_{\varepsilon}:=\gamma_{\eta_2/2}\,,\,\forall\,\eta_2\leq%
\varepsilon\leq 1\,$. So $\gamma=(\gamma_{\varepsilon})$ is a closed,
simple, contractible and positively oriented history.

Let $\,z=[(z_{\varepsilon})]\in V_{\rho}[z_0]\,$. Then there is $\eta\in
I_{\eta_2}$ such that $\,\forall\,t\in\,[0,1]$ one has $\varepsilon^{%
\rho}-|z_{\varepsilon}-z_{0\varepsilon}|\geq-\varepsilon^{\rho}\,$ and $%
|z_{\varepsilon}-\gamma_{\varepsilon}(t)|=|z_{\varepsilon}-z_{0%
\varepsilon}-re^{2\pi i t}|\geq r-|z_{\varepsilon}-z_{0\varepsilon}| \geq
r-2\varepsilon^{\rho}\geq r-2\rho>2\rho\,.$ Thus $\,d(z,\gamma^*)\in Inv({%
\mathbb{C}})$. Fix $\varepsilon\in\mathbf{I}_{\eta}\,$, $\,w\in\gamma_{%
\varepsilon}([0,1])$ and note that

\begin{equation*}
{\frac{f_{\varepsilon}(w)}{{w-z_{\varepsilon}}}}={\frac{f_{\varepsilon}(w)}{{%
w-z_{0\varepsilon}}}}\sum_{n=0}^{\infty}\left ({\frac{z_{\varepsilon}-z_{0%
\varepsilon}}{{w-z_{0\varepsilon}}}}\right )^n=\sum_{n=0}^{\infty}
f_{\varepsilon}(w){\frac{(z_{\varepsilon}-z_{0\varepsilon})^n\,\,\,}{{%
(w-z_{0\varepsilon})^{n+1}}}}\,,
\end{equation*}
and so 
\begin{equation*}
2\pi i f_{\varepsilon}(z_{\varepsilon})=\int_{\gamma_{\varepsilon}}{\frac{%
f_{\varepsilon}(w)}{{w-z_{\varepsilon}}}}\,dw=\sum_{n=0}^{\infty}
\int_{\gamma_{\varepsilon}}f_{\varepsilon}(w){\frac{(z_{\varepsilon}-z_{0%
\varepsilon})^n\,\,\,}{{(w-z_{0\varepsilon})^{n+1}}}}\,dw\,.
\end{equation*}
Thus

{\small 
\begin{eqnarray*}
2\pi i(\kappa (f))(z)=\int_{\gamma}{\frac{f(w)}{{w-z}}}\,dw&=&\left [\left
(\,\int_{\gamma_{\varepsilon}}{\frac{f_{\varepsilon}(w)}{{w-z_{\varepsilon}}}%
}\,dw\,\right )\right ]\cr\cr\cr &=&\left [\left (\,\sum_{n=0}^{\infty}
\int_{\gamma_{\varepsilon}}f_{\varepsilon}(w){\frac{(z_{\varepsilon}-z_{0%
\varepsilon})^n\,\,\,}{{(w-z_{0\varepsilon})^{n+1}}}}\,dw\,\right )\right
]\,.\cr
\end{eqnarray*}%
}

\vskip-0.3cm \noindent Using that $\,||z-z_0||<1\,$, $\gamma^*$ is a
pre-membrane and $(f_{\varepsilon})$ is moderate it is not difficult to
prove that $\sum_{n=0}^{\infty} \int_{\gamma}f(w){\frac{(z-z_0)^n\,\,\,}{{%
(w-z_{0})^{n+1}}}}\,dw\,$ converges (since that $\lim\limits_{n\to
\infty}\int_{\gamma}f(w){\frac{(z-z_0)^n\,\,\,}{{(w-z_0)^{n+1}}}}\,dw=0\,$)
and that {\small 
\begin{equation*}
\sum_{n=0}^{\infty} \int_{\gamma}f(w){\frac{(z-z_0)^n\,\,\,}{{(w-z_{0})^{n+1}%
}}}\,dw\,=\left [\left ( \sum_{n=0}^{\infty}
\int_{\gamma_{\varepsilon}}f_{\varepsilon}(w){\frac{(z_{\varepsilon}-z_{0%
\varepsilon})^n\,\,\,}{{(w-z_{0\varepsilon})^{n+1}}}}\,dw\,\right)\right ]\,.
\end{equation*}%
} Hence $\,2\pi i(\kappa (f))(z)=\sum_{n=0}^{\infty} \int_{\gamma}{\frac{%
\,f(w)\,\,\,}{{\,(w-z_{0})^{n+1}}}}\,dw\,(z-z_0)^n\,,\,\forall\,z=[(z_{%
\varepsilon})]\in V_{\rho}[z_0]\,$.
\end{proof}


\section{\textbf{\ The Transport Equation}}

In this section we shall consider the transport equation with generalized
coefficients. We prove that we have now all the tools to give a classical
solution of this problem.

Let $\Omega={\mathbb{R}}^n\times\,]0,\infty[\,$ and $u:{\tilde{\Omega}}%
_c\rightarrow{\overline{{\mathbb{R}}}}$ be a differentiable function. Denote
by $u_t$ the partial derivative of $u$ in the last variable and let  $\nabla
u=(\nabla_xu,u_t)$, where $\nabla_xu$ is the gradient vector of $u$ with
respect to the first $n$ variables. For $b\in {\overline{{\mathbb{R}}}}^n$
and $g\in\mathcal{C}^{1}(\widetilde{{\mathbb{R}}^n}_c,{\overline{{\mathbb{R}}%
}})\,$ we consider the transport equation

\vskip-0.2cm  
\begin{equation}  \label{transp}
\ u_t+\langle \, \nabla_xu\,|\,b\,\rangle=0\,\,\hbox {in } \,{\tilde{\Omega}}%
_c\,\,\,,\,\,\,u(x,0)=g(x)\,,\,\forall\,x\in\widetilde{{\mathbb{R}}^n}_c\,.
\end{equation}

Let $\,w(x,t):=g(x-tb)\,$, $\,\forall\,(x,t)\in {\tilde{\Omega}}_c\,$. By
Theorem~\ref{curves}, we have that $\,w\,$ is a solution of (1). Moreover,
if $g\in\kappa(\mathcal{G}({\mathbb{R}}^n))$ and $\,v\in\mathcal{G}(\Omega)$
is a solution of (1) in $\mathcal{G}(\Omega)$, then $\,\kappa v=w\,$.

Just like in the classical case, we can get an explicit solution of the
boundary value problem

\vskip-0.4cm 
\begin{equation}  \label{transp1}
\ u_t+\langle \,\nabla_xu\,|\,b\,\rangle=\kappa (f)\,\,\hbox {in } \,{\tilde{%
\Omega}}_c\,\,\,,\,\,\,u(x,0)=g(x)\,,\,\forall\,x\in\widetilde{{\mathbb{R}}^n%
}_c\,,
\end{equation}

\noindent where $\,f\in\mathcal{G}({\mathbb{R}}^n\times\,]-a,\infty[)$ for
some $\,a>0\,$. In this case, let \vskip-0.2cm 
\begin{equation*}
\,w(x,t):=
g(x-tb)+\int_{M_{t}}\,f(x+sb,t+s)\,d\lambda(s)\,,\,\forall\,(x,t)\in{\tilde{%
\Omega}}_c\,,
\end{equation*}
where $\,\lambda$ is the Lebesgue measure on ${\mathbb{R}}\,$ and if $%
\,t=[(t_{\varepsilon})]$, then $\,M_{t}=[\,(\,[-t_{\varepsilon},0]\,)\,]\,$.
So $w$ is a solution of (2). Besides, if $g\in\kappa(\mathcal{G}({\mathbb{R}}%
^n))$ and $\,v\in\mathcal{G}(\Omega)$ is a solution of (2) in $\mathcal{G}%
(\Omega)$, then $\,\kappa v=w\,$.

\vspace{0.3cm} This proves that the differential calculus we developed
allows us to solve the generalized transport equation just like in the
classical case. Using the solution of this equation, we can solve, for $%
\,n=1\,$ and $\,g\,,\,h\in\mathcal{G}({\mathbb{R}}^n)\,$, the boundary value
problem

\vskip-0.6cm 
\begin{equation*}
u_{tt}- u_{xx}=0\,\,\hbox{ in }\,\, {\tilde{\Omega}}_c\,\,\,,\,\, u(x,0)=
(\kappa g)(x)\,\,\hbox {and }\,\, u_t(x,0)= (\kappa h)(x)
\,\,,\,\forall\,x\in\widetilde{{\mathbb{R}}^n}_c\,,
\end{equation*}
giving a formula for its solution just as is done in the classical case. In
this case a solution is the function 
\begin{equation*}
w(x,t)={\frac{1}{{2}}}[\,g(x+t)+g(x-t)\,] +{\frac{1}{{2}}}%
\int_{M_{x\,t}}\,h\,d\lambda\,,\,\forall\, (x,t)\in{\tilde{\Omega}}_c\,,
\end{equation*}
where $\,\lambda$ is the Lebesgue measure on ${\mathbb{R}}\,$ and if $%
x=[(x_{\varepsilon})]\,$ and $\,t=[(t_{\varepsilon})]$, then $%
\,M_{x\,t}=[\,(\,[x_{\varepsilon}-t_{\varepsilon},x_{\varepsilon}+t_{%
\varepsilon}]\,)\,]\,$. \vspace{0.6cm}

\noindent {\underline{\textbf{Acknowledgment:}} Part of this work was done
when the last author was visiting the University of S\~ao Paulo in August
2006. His visit was partially supported by FAPESP-Brazil. One should also
notice that in \cite{Eve} there is also a version of the implicit function
theorem and that in \cite{ever} membrane are treated in a more general
sense. }

\noindent J. Aragona, R. Fernandez and S.O. Juriaans\newline
Instituto de Matem\'atica e Estat\'\i stica - Universidade de S\~ao Paulo%
\newline
CP 66281 - CEP 05311-970 - S\~ao Paulo - Brazil\newline
E-mail addresses: aragona@ime.usp.br;\,\, roselif@ime.usp.br;\newline
ostanley@ime.usp.br

\vspace{0.6cm}

\noindent M. Oberguggenberger\newline
Institut f$\ddot{u}$r Technische Mathematik, Geometrie und Bauinformatik,
Universit$\ddot{a}$t Inssbruck, A - 6020 Inssbruck, Austria\newline
E-mail address: michael@mat1.uibk.ac.at

\end{document}